\documentclass[10pt,a4paper]{amsart}

\usepackage{graphicx,amssymb,color,amscd}
\usepackage{xfrac}    

\usepackage[pdftex,all]{xy}

\theoremstyle{plain}
\newtheorem{theorem}{Theorem}[section]
\newtheorem{lemma}[theorem]{Lemma}
\newtheorem{proposition}[theorem]{Proposition}
\newtheorem{corollary}[theorem]{Corollary}
\theoremstyle{definition}
\newtheorem{definition}[theorem]{Definition}
\newtheorem{thm}{Theorem}
\newtheorem{prop}{Proposition}
\theoremstyle{remark}

\newtheorem{remark}[theorem]{Remark}

\newtheorem{claim}[theorem]{Claim}
\newtheorem*{acknowledgments}{Acknowledgments}
\numberwithin{equation}{section}
\numberwithin{figure}{section}

\newcommand{\bd}{\begin{description}}   
\newcommand{\ed}{\end{description}} 
\newcommand{\ba}{\begin{array}}      \newcommand{\ea}{\end{array}} 
\newcommand{\bc}{\begin{center}}     \newcommand{\ec}{\end{center}} 
\newcommand{\be}{\begin{enumerate}}  \newcommand{\ee}{\end{enumerate}} 
\newcommand{\beq}{\begin{eqnarray}}  \newcommand{\eeq}{\end{eqnarray}} 
\newcommand{\beQ}{\begin{eqnarray*}} \newcommand{\eeQ}{\end{eqnarray*}} 
\newcommand{\bi}{\begin{itemize}}    \newcommand{\ei}{\end{itemize}}

\newcommand{\bfrac}[2]{{\raisebox{.1em}{$#1$} / \raisebox{-.3em}{$#2$}}}

\begin{document} 
\title[Link concordances as surfaces in $4$-space]{Link concordances as surfaces in $4$-space and the 4-dimensional Milnor invariants}
\author[J.B. Meilhan]{Jean-Baptiste Meilhan} 
\address{Univ. Grenoble Alpes, CNRS, IF, 38000 Grenoble, France}
	 \email{jean-baptiste.meilhan@univ-grenoble-alpes.fr}
\author[A. Yasuhara]{Akira Yasuhara} 
\address{Faculty of Commerce, Waseda University, 1-6-1 Nishi-Waseda, Shinjuku-ku, Tokyo 169-8050, Japan}
	 \email{yasuhara@waseda.jp}
\date{\today}
\begin{abstract}
Fixing two concordant links in $3$--space, we study the set of all embedded concordances between them, as knotted annuli in $4$--space. 
When regarded up to surface-concordance or link-homotopy,  the set $\mathcal{C}(L)$ of concordances from a link $L$ to itself forms a group. 
In order to investigate these groups, we define Milnor-type invariants of $\mathcal{C}(L)$, which are integers defined modulo a certain indeterminacy given by Milnor invariants of $L$. 
We show in particular that, for a slice link $L$,  these invariants classify $\mathcal{C}(L)$ up to link-homotopy. 
\end{abstract}

\maketitle
 \begin{center}
   \normalsize \em {Dedicated to the memory of Nathan Habegger}
 \end{center}
\bigskip

\section*{Introduction}

Throughout this paper, we work in the smooth category. 

Let $L$ and $L'$ be two ordered and oriented $n$-component links in the interior of the $3$-ball $B^3$.
The links $L$ and $L'$ are \emph{concordant} if there exists an embedding 
$C: \sqcup_{i} (S^1\times [0,1])_i\hookrightarrow B^3\times [0,1]$ of $n$ disjoint oriented annuli into  
$B^3\times [0,1]$ such that, for each $i$, the image of the $i$th copy of $S^1\times [0,1]$ 
intersects $B^3\times \{0\}$ along 
the $i$th component of $L$, and $B^3\times \{1\}$ along 
the $i$th component of $L'$, respecting the orientation. 
The image of such an embedding is called a \emph{concordance from $L$ to $L'$}.
Concordance defines an equivalence relation on the set of $n$-component links, 
which is widely studied in the literature. 

In this paper, we study link concordance from a different point of view, 
by considering the set of all concordances between two given links. 
We denote by $\mathcal{C}(L,L')$ the set of all concordances from $L$ to $L'$. 
Clearly, $\mathcal{C}(L,L')$ is nonempty if and only if the links $L$ and $L'$ are concordant. 

We will also consider the larger set $\mathcal{C}^s(L,L')$ of \emph{self-singular concordances} from $L$ to $L'$,  
that is, concordances where we allow each component to self-intersect  but where distinct components are disjoint.
Note that $\mathcal{C}^s(L,L')$ is nonempty if and only if the links $L$ and $L'$ are link-homotopic, 
i.e. are related by a sequence of self-crossing changes and isotopies, see \cite{Giffen,Goldsmith} and also \cite{H}.

We we shall sometimes use the term \lq embedded concordance\rq\, when referring specifically to elements of  $\mathcal{C}(L,L')$ which are \emph{not} in $\mathcal{C}^s(L,L')$. 

In the case $L=L'$, we simply denote $\mathcal{C}(L)=\mathcal{C}(L,L)$ and $\mathcal{C}^s(L)=\mathcal{C}^s(L,L)$. 
\medskip 

Let $L$, $L'$ and $L''$ be three concordant links. 
For $C\in \mathcal{C}(L,L')$ and $C'\in \mathcal{C}(L',L'')$, identifying the upper boundary of $C$ with the lower boundary of $C'$ yields the \emph{stacking product} $C\cdot C'$, which is an element of $\mathcal{C}(L,L'')$. 
The stacking product is of course more generally defined for self-singular concordances. 
\medskip

\subsection*{Equivalence relations} 
Two (self-singular) concordances are \emph{equivalent} if they are related by an ambient isotopy fixing the boundary. 
Note that this means in particular that $L$ and $L'$ are \emph{fixed} links in $B^3$, hence cannot be deformed by an ambient isotopy of $B^3$. 

The stacking product endows the set of equivalence classes of $\mathcal{C}^s(L)$ with a structure of monoid, containing those of $\mathcal{C}(L)$ as a submonoid, where the identity element is given by $L\times [0,1]$. 
In the special case where $L$ is the boundary $O_n$ of a disjoint union of $n$  Euclidian disks in a plane of $B^3$,  
then the set of equivalence classes of $\mathcal{C}(O_n)$ is the monoid of \emph{$2$-string links} introduced and studied in \cite{AMW}. 
\medskip

In this paper, we also consider the following two natural equivalence relations on (self-singular) concordances. 

Two elements $C,C'$ of $\mathcal{C}(L,L')$ are 
\emph{(surface-)concordant} if there exists an embedding 
$W: \sqcup_{i} \left((S^1\times [0,1])\times [0,1]\right)_i\hookrightarrow (B^3\times [0,1])\times [0,1]$
such that, for each $i$, the image of the $i$th copy of $(S^1\times [0,1])\times [0,1]$ 
intersects $(B^3\times [0,1])\times \{0\}$, resp. $(B^3\times [0,1])\times \{1\}$, along the $i$th component of $C$, resp. $C'$, respecting the orientation, and such that 
$W\left( (S^1\times \{0\})\times [0,1]\right) = (L\times \{0\})\times [0,1]$  
and 
$W\left( (S^1\times \{1\})\times [0,1]\right) = (L'\times \{1\})\times [0,1]$.
This defines an equivalence relation on $\mathcal{C}(L,L')$, which we denote by \raisebox{-.2em}{$\stackrel{c}{\sim}$}. 
In what follows, we will often blur the distinction between the embedding defining a surface-concordance, and its image.  

Two elements of $\mathcal{C}^s(L,L')$ are \emph{(surface-)link-homotopic} 
if they are homotopic through self-singular concordances, that is, 
if there exists a continuous deformation of one into the other 
such that distinct components remain disjoint at all time, but where self-intersections may occur. 
This defines another equivalence relation on $\mathcal{C}^s(L,L')$ and $\mathcal{C}(L,L')$, which we denote by \raisebox{-.2em}{$\stackrel{lh}{\sim}$}. 

We observe that the stacking product endows the quotient sets $\bfrac{\mathcal{C}(L)}{\stackrel{c}{\sim}}$,  
  $\bfrac{\mathcal{C}(L)}{\stackrel{\small{lh}}{\sim}}$  and $\bfrac{\mathcal{C}^s(L)}{\stackrel{\small{lh}}{\sim}}$ with a structure of monoid. 

The terminologies for these two equivalence relations emphasize the fact that these are relations for surfaces in $4$-space. 
This avoids possible confusions with the analogous relations for classical links, which are also involved in the discussion. 
In the rest of the paper, however, we shall mostly use the shorter 
terms \lq concordance\rq\, and \lq link-homotopy\rq\, for knotted surfaces when the context is explicit. 
\medskip 

For simplicity of  exposition, in the rest of this introduction we restrict to the case $L=L'$, but we emphasize that, in this paper, the constructions and most of the results are given in the more general context of two concordant links $L$ and $L'$. 

The following is mainly a combination of known facts and techniques, but it is a fundamental result: 
\begin{prop}\label{thm:groups}
Let $L$ be a link.
 \begin{itemize}
 \item[(i)]  The  monoids 
  $\bfrac{\mathcal{C}(L)}{\stackrel{c}{\sim}}$,  
  $\bfrac{\mathcal{C}(L)}{\stackrel{\small{lh}}{\sim}}$  
  and $\bfrac{\mathcal{C}^s(L)}{\stackrel{\small{lh}}{\sim}}$
  are all groups.
  \item[(ii)]  For any link $L'$ concordant to $L$, 
  there is an isomorphism $\bfrac{\mathcal{C}(L)}{\stackrel{c}{\sim}}\simeq \bfrac{\mathcal{C}(L')}{\stackrel{c}{\sim}}$. 
  \item[(iii)]  For any link $L'$ link-homotopic to $L$, 
  there is an isomorphism $\bfrac{\mathcal{C}^s(L)}{\stackrel{lh}{\sim}}\simeq \bfrac{\mathcal{C}^s(L')}{\stackrel{lh}{\sim}}$, which restricts to an isomorphism 
  $\bfrac{\mathcal{C}(L)}{\stackrel{lh}{\sim}}\simeq \bfrac{\mathcal{C}(L')}{\stackrel{lh}{\sim}}$. 
  \end{itemize}
\end{prop}
This paper addresses the problem of classifying these quotients, by defining and studying some invariants for embedded concordances. 
\medskip

\subsection*{Artin-type invariants}
We first define representations of the concordance and link-homotopy groups of embedded concordances, that are modelled on Artin's representation of the braid group as automorphisms of the free group.

For each $k\ge 1$, denote by $N_k G=G / \Gamma_k G$ 
the \emph{$k$th nilpotent quotient} of a group  $G$, 
where $\Gamma_k G$ is the $k$th term of the \emph{lower central series} of $G$, 
which is defined inductively by $\Gamma_1 G=G$ and $\Gamma_{k+1}G = [G,\Gamma_k G]$. 

Let $C\in \mathcal{C}(L)$.  
Denote by $\pi_L$ the fundamental group of $B^3\setminus L$; recall that the $k$th nilpotent quotient of $\pi_L$ is generated by $n$ elements, given by the choice of a meridian for each component \cite{Chen}. 
As developed in Section \ref{sec:phik}, by considering the natural inclusions of this link complement at the \lq top\rq~ and \lq bottom\rq~ of $(B^3\times [0,1])\setminus C$, and by using a theorem of Stallings \cite[Thm.~5.1]{Stallings}, we obtain for each $k\ge 1$ a map 
  $$ \varphi^{L}_k: \mathcal{C}(L)\longrightarrow \textrm{Aut}\left(N_k \pi_L \right). $$
More precisely, $\varphi^L_k$ takes values in the subgroup $\textrm{Aut}_\textrm{c}\left(N_k \pi_L\right)$ of automorphisms acting by \lq conjugation on each meridian\rq~ -- see Section \ref{sec:longitude}. 

\begin{thm}\label{thm:surj}
 Let $k\ge 1$ be an integer. 
\begin{itemize}
 \item[(i)] The map $\varphi^{L}_k$ is a surface-concordance invariant. 
 \item[(ii)] 
 If $L$ is slice, then the induced map  
 $\bfrac{\mathcal{C}(L)}{\stackrel{\small{c}}{\sim}}\rightarrow \textrm{Aut}_\textrm{c}\left(N_k \pi_L \right)$ is a surjective  homomorphism. 
\end{itemize}
\end{thm}
Note that, if $L$ is a slice link, then $N_k \pi_L$ is isomorphic to $N_kF_n$, the $k$th nilpotent quotient of the free group $F_n$ on $n$ elements \cite{Milnor2}.  
\medskip

In order to obtain a link-homotopy invariant, we consider reduced groups. 
Given a group $G$ normally generated by $g_1,\cdots,g_k$, the \emph{reduced group} $\textrm{R}G$ is the quotient by the normal closure of all relations $[g_i, x^{-1} g_i x]$, where $i\in\{1,\cdots,k \}$ and $x\in G$.  
In particular, R$\pi_L$ is the  largest quotient of $\pi_L$ where each meridian commutes with any of its conjugates.
A similar construction as above then yields a map 
  $$ \varphi^{L}_{\textrm{R}}: \mathcal{C}(L)\longrightarrow \textrm{Aut}_\textrm{c}\left(\textrm{R} \pi_L \right). $$
  
\begin{thm}\label{thm:linkhom}
\begin{itemize}
 \item[ ]
 \item[(i)] The map $\varphi^{L}_{\textrm{R}}$ is a link-homotopy invariant. 
 \item[(ii)] 
 If $L$ is slice, then 
 the induced map  
 $\bfrac{\mathcal{C}(L)}{\stackrel{\small{lh}}{\sim}}\rightarrow \textrm{Aut}_\textrm{c}\left(\textrm{R}\pi_L\right)$ is an isomorphism, 
 i.e.  $\varphi^{L}_{\textrm{R}}$ classifies $\mathcal{C}(L)$ up to link-homotopy.
\end{itemize}
\end{thm}
This generalizes a result obtained in the special case  $L=O_n$ in \cite{AMW,ABMW}. 
Note that, for a slice link $L$, R$\pi_L$ is isomorphic to the reduced free group R$F_n$. 
\medskip

\subsection*{Milnor-type invariants}

Next, we define numerical invariants of embedded concordances, that are modeled on Milnor invariants of links in  $3$-space (reviewed in Section \ref{sec:review}).

Let $C\in \mathcal{C}(L)$.
For each component $C_i$ of $C$, there is a notion of \emph{preferred longitude}, which is represented by an arc running \lq parallel to $C_i$\rq,~ from bottom to top. 
We emphasize however that, unlike in the classical link case, there is not a unique preferred longitude for $C_i$. 
As explained in Section \ref{sec:Milnor}, a preferred longitude can be expressed 
as an element in $N_k \pi_{L}$,  
and taking the \emph{Magnus expansion} of this element yields a formal power series in non commuting variables $X_1,\cdots,X_n$. 
Denote by $\mu^{(4)}_C(i_1\cdots i_{m} i)$ the coefficient of $X_{i_1}\cdots X_{i_{m}}$ in this power series. 
\emph{Milnor $\overline{\mu}^{(4)}$-invariants} of the concordance $C$ are given by the following: 
\begin{thm}\label{thm:invmu}
Let $C\in \mathcal{C}(L)$ be an embedded concordance.
Let $I$ be any sequence of at most $k$ elements in $\{1,\cdots,n\}$. 
The residue class $\overline{\mu}^{(4)}_C(I)$ of $\mu^{(4)}_C(I)$ modulo $\overline{\Delta}_{L}(I)$ 
is an (ambient isotopy relative to the boundary) invariant of $C$.
\end{thm}
\noindent Here, the indeterminacy $\overline{\Delta}_L(I)$ is given by Milnor invariants of the \emph{link} $L$ at the boundary. 
This is essentially due to the fact that different choices of preferred longitude for the $i$th component of $C$ differ by 
parallel copies of the $i$th component of the link $L$: via the Magnus expansion, these copies introduce extra terms which are all Milnor numbers of $L$. 
In particular, if $L$ is a slice link, then all Milnor invariants of $L$ vanish and the integer $\mu^{(4)}_C(I)$ is an invariant of $C$; see Remark \ref{rem:ark}. 
This interplay between $3$-dimensional and $4$-dimensional invariants seems quite remarkable.

Properties of these 
$\overline{\mu}^{(4)}$-invariants of embedded concordances are investigated. 
In particular, we have: 
\begin{thm}\label{thm:mu}
\begin{enumerate}
\item[ ]
\item[(i)] Milnor $\overline{\mu}^{(4)}$-invariants are surface-concordance invariants. 
\item[(ii)] Milnor $\overline{\mu}^{(4)}$-invariants indexed by non-repeated sequences are link-homotopy invariants. 
\item[(iii)] If $L$ is slice, Milnor 
$\mu^{(4)}$-invariants indexed by non-repeated sequences form a complete link-homotopy invariant for $\mathcal{C}(L)$. 
In other words, two embedded concordances from $L$ to itself are link-homotopic if and only if 
all their Milnor $\mu^{(4)}$-invariants indexed by non-repeated sequences coincide. 
\end{enumerate}
\end{thm}

It seems highly unlikely that Theorem \ref{thm:mu}~(iii) holds for an arbitrary link. For instance, if $L$ is a positive Hopf link, 
then elements of $\mathcal{C}(L)$ are undistinguisable by $\overline{\mu}^{(4)}$-invariants, since $\overline{\Delta}_{L}(I)=1$ for any sequence $I$ with at least $2$ indices. 
\medskip

Some applications of these numerical invariants are given in Section \ref{sec:appli}, involving simple examples of explicit computations. 
In particular, it is shown that for a slice link $L$, the groups $\bfrac{\mathcal{C}(L)}{\stackrel{c}{\sim}}$ and   
  $\bfrac{\mathcal{C}(L)}{\stackrel{\small{lh}}{\sim}}$ are not abelian, except for very small numbers of components (Prop.~\ref{proposition1} and \ref{proposition2}). 
  Since $\mathcal{C}(L)\subset \mathcal{C}^s(L)$, this implies that $\bfrac{\mathcal{C}^s(L)}{\stackrel{\small{lh}}{\sim}}$ is not abelian either. 
  \\
  More can actually be said about the group $\bfrac{\mathcal{C}(L)}{\stackrel{\small{lh}}{\sim}}$ in the slice case. 
  Indeed by Theorem \ref{thm:linkhom} this group coincides with the group Aut$_c(\textrm{R}F_n)$, also known as the \emph{reduced McCool group}, 
  or group of \emph{welded string links up to homotopy} studied in \cite{ABMW,AMW,D}. This is a nilpotent group (of class $n-1$ for an $n$ component link $L$), 
  whose Hirsch rank is given in \cite[Rem~4.9]{AMW} (see also \cite[Cor.~2.11]{D}); 
  a (finite) presentation for this group was recently computed in \cite[Thm.~5.8]{D}.  
  We expect that the tools of the present paper might lead to a better understanding of these groups in the general case.  

\begin{acknowledgments}
The authors are indebted to Benjamin Audoux for stimulating discussions in the course of the preparation of this paper, and in particular 
for pointing out that the notion of R-coloring can be used to prove the link-homotopy invariance result. They also thank Jacques Darn\'e, Mark Powell and the referee for insightful comments. \\
The second author was supported by JSPS KAKENHI Grant Number JP21K03237 and a Waseda University Grant for Special Research Projects (Project number: 2020R-1018). This work was also partly supported by the project AlMaRe ANR-19-CE40-0001-01.
\end{acknowledgments}

\section{Groups of concordance and link-homotopy classes}

In this short section, we prove a slightly generalized version of Proposition \ref{thm:groups}. 

The following  readily implies Proposition \ref{thm:groups}~(i). 
More generally this lemma endows (self-singular) concordances up to link-homotopy with a structure of groupoid: the category with objects given by the set of all links of $B^3$ and with morphisms given by the (possibly empty) set of link-homotopy classes of (self-singular) concordances between two given links, has the property that every morphism is invertible. 
\begin{lemma}\label{lem:inverse}
For a (self-singular) concordance $X$, we denote by $\overline{X}$ its image by the involution of $B^3\times [0,1]$ mapping $B^3\times \{t\}$ to $B^3\times \{(1-t)\}$ for all $t$.
  \begin{enumerate}
  \item For any concordance $C$ in $\mathcal{C}(L,L')$, we have 
    $$ C\cdot \overline{C}\stackrel{c}{\sim} L\times [0,1]\quad\textrm{and}\quad \overline{C}\cdot C\stackrel{c}{\sim} L'\times [0,1]. $$
  \item For any self-singular concordance (and, in particular, any concordance) $S$ in $\mathcal{C}^s(L,L')$, we have
    $$ S\cdot \overline{S}\stackrel{lh}{\sim} L\times [0,1]\quad\textrm{and}\quad \overline{S}\cdot S\stackrel{lh}{\sim} L'\times [0,1]. $$
  \end{enumerate}
\end{lemma}
\begin{proof}
We first give the argument for (1) -- which is essentially the same proof as for the analoguous fact one dimension down. Given a (self-singular) 
concordance $C$ from $L$ to $L'$, consider the product $C\times [0,1]$ in $(B^3\times [0,1])\times [0,1]\simeq B^5$. Then the  boundary of $C\times [0,1]$ is 
$C\cup (L\times [0,1])\cup \overline{C}\cup (L'\times [0,1])$.  Contracting $L'\times [0,1]$ to $L'\times \{0\}$, we thus obtain an explicit concordance from $C\cdot \overline{C}$ to $L\times [0,1]$. 
\\
Then (2) follows, by a general result of Bartels and Teichner \cite[Thm.~5]{BT}, 
stating that \lq (self-singular) concordance implies link-homotopy\rq~ for surface-links
\end{proof}
\medskip

Next, we introduce two maps, defined by stacking operations, and which encode the interplay between the sets 
of (self-singular) concordances when we let the boundary link vary. 
\begin{definition}\label{def:maps}
Let $X\in \mathcal{C}^s(L,L')$ be a self-singular concordance. 
We define two maps 
\begin{eqnarray*}
 \xi_X:  \mathcal{C}^s(L')\longrightarrow \mathcal{C}^s(L) \,& ; & \,C'\mapsto X\cdot C'\cdot \overline{X}
 \end{eqnarray*}
and 
\begin{eqnarray*}
\eta_X: \mathcal{C}^s(L)\longrightarrow \mathcal{C}^s(L,L') \, & ; &   \, C\mapsto C\cdot X.
\end{eqnarray*}
\end{definition}

We have the following, which in particular implies Proposition \ref{thm:groups}~(ii) and (iii).  
\begin{lemma}\label{thm:iso}
Let $X\in \mathcal{C}^s(L,L')$ be a self-singular concordance. 
The maps $\xi_X$ and $\eta_X$ descend, respectively, to a group isomorphism and a bijection 
$$  
\xi^{lh}_X:  \bfrac{\mathcal{C}^s(L')}{\stackrel{lh}{\sim}}\longrightarrow \bfrac{\mathcal{C}^s(L)}{\stackrel{lh}{\sim}} 
\,\, \textrm{ and }\,\,  
\eta^{lh}_X: \bfrac{\mathcal{C}^s(L)}{\stackrel{lh}{\sim}}\longrightarrow \bfrac{\mathcal{C}^s(L,L')}{\stackrel{lh}{\sim}}. 
$$
Moreover, if $X$ is an embedded concordance, i.e. $X\in \mathcal{C}(L,L')$, we obtain likewise 
group isomorphisms 
$$\xi^{c}_X:  \bfrac{\mathcal{C}(L')}{\stackrel{c}{\sim}}\longrightarrow \bfrac{\mathcal{C}(L)}{\stackrel{c}{\sim}} 
\,\, \textrm{ and }\,\,  
 \xi^{lh}_X:  \bfrac{\mathcal{C}(L')}{\stackrel{lh}{\sim}}\longrightarrow \bfrac{\mathcal{C}(L)}{\stackrel{lh}{\sim}}, $$ 
and bijections 
$$\eta^{c}_X: \bfrac{\mathcal{C}(L)}{\stackrel{c}{\sim}}\longrightarrow \bfrac{\mathcal{C}(L,L')}{\stackrel{c}{\sim}} 
\,\, \textrm{ and }\,\,  
 \eta^{lh}_X: \bfrac{\mathcal{C}(L)}{\stackrel{lh}{\sim}}\longrightarrow \bfrac{\mathcal{C}(L,L')}{\stackrel{lh}{\sim}}. $$
\end{lemma}

\begin{proof}
This is a rather immediate consequence of Lemma \ref{lem:inverse}. 
Let us briefly outline the proof for, say, the map $\xi^{c}_X$ -- the argument is elementary, and is strictly similar for the other maps.
(We note that similar arguments appear in \cite[\S 2.5 and 2.6]{tuncador}.)  
By Lemma \ref{lem:inverse}, we have that 
$$ \xi_X(C\cdot C') = \left(X\cdot C\right)\cdot \left(C'\cdot \overline{X}\right) \stackrel{c}{\sim} \left(X\cdot C\cdot \overline{X}\right)\cdot \left(X\cdot C'\cdot \overline{X}\right) = \xi_X(C)\cdot \xi_X(C'), $$
for any $C,C'\in  \mathcal{C}(L')$. Hence $\xi^{c}_X$ is a homomorphism. 
Lemma \ref{lem:inverse} also shows that the maps $\xi^{c}_{X}$ and $\xi^{c}_{\overline{X}}$ satisfy  $\xi^{c}_{X}\circ \xi^{c}_{\overline{X}} = \textrm{Id}_{\sfrac{\mathcal{C}(L)}{\stackrel{c}{\sim}}}$ and  
$\xi^{c}_{\overline{X}}\circ \xi^{c}_{X} = \textrm{Id}_{\sfrac{\mathcal{C}(L')}{\stackrel{c}{\sim}}}$, 
which proves that $\xi^{c}_X$ is bijective.  
\end{proof}

\section{Artin-type invariants}\label{sec:artinsec}

As before, let $L$ and $L'$ be two ordered, oriented $n$-component links in the interior of the $3$-ball $B^3$. 
Fix a point $p$ in $\partial B^3$.

\subsection{Artin-type invariants of concordances}\label{sec:phik}

Let $C\in \mathcal{C}(L,L')$. 
Denote by $\pi_C$ the fundamental group of $(B^3\times [0,1])\setminus C$ based at $p\times \{0\}$. 
Denote also by $\pi_L$, resp. $\pi_{L'}$, the fundamental group of $B^3\setminus L$, resp $B^3\setminus L'$, based at $p$. 

We denote by $\iota_L$, resp. $\iota_{L'}$, the inclusion map of $B^3\setminus L$, resp $B^3\setminus L'$, in 
the complement $(B^3\times [0,1])\setminus C$.  

\subsubsection{Action on nilpotent quotients}\label{sec:artin}

Recall from the introduction that $N_k G$ denotes the $k$th nilpotent quotient of a group $G$. 
By a theorem of Stallings \cite[Thm.~5.1]{Stallings}, the inclusion maps $\iota_L$ and $\iota_{L'}$ induce isomorphisms 
\begin{equation}\label{eq:Nk}
 N_k \pi_L \stackrel{\iota_L^k}{\longrightarrow} N_k \pi_C \stackrel{\iota_{L'}^k}{\longleftarrow} N_k \pi_{L'}, 
\end{equation}
for all integers $k\ge 1$.
Hence, 
 $$ \varphi_{C,k}:= \left(\iota_{L}^k\right)^{-1}\circ \iota_{L'}^k\in \textrm{Iso}\left(N_k \pi_{L'},N_k \pi_{L} \right) $$
is an invariant of (ambiant isotopy relative to the boundary of) $C$.

\begin{definition}\label{def:phik}
 For all $k\ge 1$, we denote by
  $$ \varphi^{L,L'}_k: \mathcal{C}(L,L')\longrightarrow \textrm{Iso}\left(N_k \pi_{L'},N_k \pi_{L} \right)$$
  the map defined by $ \varphi^{L,L'}_k(C) := \varphi_{C,k}$. 
\end{definition}
\noindent (Note that, as claimed above, this map actually factors through the monoid of equivalence classes of $\mathcal{C}(L)$. )

\subsubsection{Meridians and longitudes}\label{sec:longitude}

The map $\varphi^{L,L'}_k$ of Definition \ref{def:phik} can be described somewhat more precisely. 

Recall  from \cite{Chen} that the nilpotent quotient $N_k \pi_L$, resp. $N_k \pi_{L'}$, is generated by $n$ elements, which are given by the \emph{choice} of a meridian for each component of $L$, resp. $L'$.
Although the main definitions and most results of this paper hold for a general notion of meridian, we will restrict ourselves to the following definition to avoid technical issues.\footnote{Specifically, Definition \ref{def:merid} is mainly only necessary in Sections \ref{sec:lhcol} and \ref{sec:lhcol2}. }

For each $i$, pick a fixed point $p_{i,0}$, resp. $p_{i,1}$, near the $i$th component $L_i$ of $L$, resp. $L'_i$ of $L$. 
Denote by $v_p$ the unit normal vector at the basepoint $p\in S^2=\partial B^3$. 
\begin{definition}\label{def:merid}
A \emph{meridian} for the $i$th component $L_i$ of $L$, resp. $L'_i$ of $L'$, 
is the union of 
 a small loop based at $p_{i,0}$, resp $p_{i,1}$, enlacing this component positively 
 and an arc that first runs from $p_{i,0}$, resp $p_{i,1}$, to $S^2=\partial B^3$ following the $v_p$ direction, then to the basepoint $p$ following a geodesic arc in $S^2$. 
A \emph{system of meridians} for $L$, resp. $L'$ is a union of meridians, one for each component of $L$, resp. $L'$. Up to a small perturbation of the basepoint, we may assume that these meridians are disjoint from $L$, resp. $L'$, and are mutually disjoint except at $p$. 
A \emph{basing} for $C\in \mathcal{C}(L,L')$ is the choice of a system of meridians for both $L$ and $L'$. 
\end{definition}

Suppose for now that elements of $\mathcal{C}(L,L')$ are equipped with a common fixed basing  
(and in particular with fixed points $p_{i,0}$ and $p_{i,1}$). 
Denote respectively by $\{m_1,\cdots,m_n\}$ and $\{m'_1,\cdots,m'_n\}$ the generating set for 
$N_k \pi_L$ and $N_k \pi_{L'}$ given by such a basing. 
Then $\varphi_{C,k}$ sends each $m'_i$ to a conjugate of $m_i$; $1\le i\le n$. In other words, there exists elements 
$l_1,\cdots,l_n$ in $N_k \pi_{L}$ such that 
\begin{equation}\label{eq:conjug}
 \varphi_{C,k}(m'_i) = l_i^{-1} m_i l_i \textrm{; $1\le i\le n$.} 
\end{equation}
\noindent 
This conjugating element $l_i\in N_k \pi_{L}$ is the image by $\left(\iota_{L}^k\right)^{-1}$ of an element of 
$ N_k \pi_{C}$ representing an $i$th longitude for $C$, which is defined as follows: 
\begin{definition}\label{def:long}
Let $C\in \mathcal{C}(L,L')$  be equipped with a basing. 
An \emph{$i$th longitude} of $C$ is an arc embedded in the boundary of a regular neighborhood of the $i$th component of $C$ and running from the fixed point $p_{i,0}$ to the fixed point $p_{i,1}$. 
Such an arc can be canonically closed into a loop using the arcs from $p_{i,0}$ to $p\times \{0\}$ and from $p_{i,1}$ to $p\times \{1\}$ given by the basing, and the arc $p\times [0,1]$. 
\end{definition}
Note that $m_i$ and $m'_i$ are conjugate by any choice of $i$th longitude for $C$ -- 
see Section \ref{sec:4dimM} for more on longitudes.
\medskip 

Now, it is not difficult to check that an element of Iso$\left(N_k \pi_{L'} ,N_k \pi_{L} \right)$ which \lq acts by conjugation\rq~ in the sense of Equation (\ref{eq:conjug}) for a given basing, also acts by conjugation on any other choice of basing. 
Hence it  make sense to define the subset Iso$_\textrm{c}\left(N_k \pi_{L'} ,N_k \pi_{L} \right)$ of 
Iso$\left(N_k \pi_{L'} ,N_k \pi_{L} \right)$ of elements acting by conjugation on any choice of basing for $C$. 

The above discussion shows that we actually have, for all $k\ge 1$, a map
  $$ \varphi^{L,L'}_k: \mathcal{C}(L,L')\longrightarrow \textrm{Iso}_\textrm{c}\left(N_k \pi_{L'} ,N_k \pi_{L} \right).$$

Notice that, in the case $L=L'$, the map $\varphi^{L}_k:=\varphi^{L,L}_k$ thus takes its values in the subgroup Aut$_\textrm{c}\left(N_k \pi_L\right)$ of automorphisms of $N_k \pi_L$ acting by conjugation on each generator, for \emph{any} choice of system of meridians for $L$. 

\subsubsection{Action on reduced groups}\label{sec:reduced} 

Recall that R$\pi_L$ is the largest quotient group of $\pi_L$ where each meridian  $m_i$ commutes with any of its conjugates. 

By  \cite[Lem.~5]{Milnor}, the reduced groups R$\pi_L$ and R$\pi_{L'}$ are nilpotent groups (of order at most $n$). 
This and (\ref{eq:Nk}) show that the inclusion maps $\iota_L$ and $\iota_{L'}$ also induce isomorphisms
\begin{equation}\label{eq:R}
 \textrm{R}\pi_L \stackrel{\iota_L^R}{\longrightarrow}  \textrm{R}\pi_C \stackrel{\iota_{L'}^R}{\longleftarrow} \textrm{R}\pi_{L'}. 
\end{equation}
Hence a group isomorphism 
 $$ \varphi_{C,R}:= \left(\iota_{L}^R\right)^{-1}\circ \iota_{L'}^R\in \textrm{Iso}\left(\textrm{R} \pi_{L'} ,\textrm{R} \pi_{L} \right). $$
\begin{definition}\label{def:phiR}
 We denote by
  $$ \varphi^{L,L'}_{\textrm{R}}: \mathcal{C}(L,L')\longrightarrow \textrm{Iso}\left(\textrm{R} \pi_{L'} ,\textrm{R} \pi_{L} \right)$$
  the map defined by $ \varphi^{L,L'}_{\textrm{R}}(C) := \varphi_{C,R}$. 
\end{definition}
\noindent As in Section \ref{sec:longitude}, this map acts by conjugation, 
that is, we actually have a map 
  $$ \varphi^{L,L'}_{\textrm{R}}: \mathcal{C}(L,L')\longrightarrow \textrm{Iso}_\textrm{c}\left(\textrm{R} \pi_{L'} ,\textrm{R} \pi_{L} \right), $$
where $\textrm{Iso}_\textrm{c}\left(\textrm{R} \pi_{L'} ,\textrm{R} \pi_{L} \right)$ denotes 
the subset of $\textrm{Iso}\left(\textrm{R} \pi_{L'} ,\textrm{R} \pi_{L} \right)$ of elements satisfying a conjugation formula of the type of (\ref{eq:conjug}) for any choice of basing for $C$. 

In the case $L=L'$, we get in this way a map 
 $$\varphi^L_{\textrm{R}}:=\varphi^{L,L}_{\textrm{R}}: \mathcal{C}(L)\rightarrow \textrm{Aut}_\textrm{c}\left( \textrm{R}\pi_L \right), $$
where, as above, the subscript c stands for automorphisms acting by conjugation on any choice of system of meridians for $L$. 

\subsection{Topological properties and classification results}

In the rest of this section, we investigate some of the properties of the invariants defined above. 
In particular, we prove slightly more general versions of Theorems \ref{thm:surj} and \ref{thm:linkhom} stated in the introduction.

In what follows, let us use the symbol $\ast$ for denoting either the letter R or any integer $k$ ($k\ge 1$). 

First note that the above-defined invariants clearly map the stacking product of two concordances to the composition of their images:  
\begin{proposition}\label{lem:additivity}
 For all links $L,L'$ and $L''$, and for all $C\in \mathcal{C}(L,L')$ and $C'\in \mathcal{C}(L',L'')$, we have 
 $\varphi^{L,L''}_\ast(C\cdot C')=\varphi^{L,L'}_\ast(C')\circ \varphi^{L',L''}_\ast(C)$.
\end{proposition}

\subsubsection{Proof of Theorem \ref{thm:surj}} 

Consider two concordant elements $C$ and $C'$ in  $\mathcal{C}(L,L')$, and pick a concordance 
$W: \sqcup_{i} \left( (S^1\times [0,1])\times [0,1] \right)_i\hookrightarrow (B^3\times [0,1])\times [0,1]$
from $C$ to $C'$.

In a similar way as in Section \ref{sec:artin}, we can consider the upper and lower inclusions of the complements of $C$ and $C'$ into the complement of $W$ in $(B^3\times [0,1])\times [0,1]$. 
Using Stallings' theorem  \cite[Thm.~5.1]{Stallings}, these inclusions induce isomorphisms at the level of the nilpotent quotients of the fundamental group (that Stallings' theorem indeed applies here is checked by a simple Mayer-Vietoris argument). 
Hence, denoting by $\pi_W$ the fundamental group of the complement of $W$, we have the following commutative diagram, for all $k\ge 1$: 
\begin{equation}
\vcenter{\hbox{\xymatrix{
    &  N_k \pi_L \ar[ld]_(.38)\simeq \ar[d]_(.38)\simeq \ar[rd]^(.38)\simeq & \\
    N_k \pi_C \ar[r]^(.38)\simeq & N_k \pi_W & N_k \pi_{C'}\ar[l]_(.38)\simeq \\
    &  N_k \pi_{L'} \ar[lu]^(.38)\simeq \ar[u]^(.38)\simeq \ar[ru]_(.38)\simeq &     
}}}.
\label{diagbis}
\end{equation}

This shows the following, which contains Theorem \ref{thm:surj}~(i). 
\begin{theorem}\label{thm:conc}
 The map $\varphi^{L,L'}_k$ is a surface-concordance invariant ($k\ge 1$). 
\end{theorem}
\begin{remark}\label{rem:concR}
 By  \cite[Lem.~1.3]{HL}, the reduced groups R$\pi_L$, R$\pi_{L'}$, R$\pi_C$, R$\pi_{C'}$ and R$\pi_W$ are all nilpotent groups of order at most $n$. 
 This shows that  (\ref{diagbis}) induces a similar commutative diagram involving these reduced groups, 
thus proving that $\varphi^{L,L'}_{\textrm{R}}$ is also a concordance invariant. Theorem \ref{thm:linkhombis} below refines this observation. 
\end{remark}

Recall that, if $L$ is a slice link, then $N_k \pi_L\simeq N_kF_n$; 
this isomorphism is not canonical, as it comes from (\ref{eq:Nk}) and thus relies upon the choice of a concordance from $L$ to $O_n$. 
The following, combined with Proposition \ref{lem:additivity}, is a generalized version of Theorem \ref{thm:surj}~(ii).
\begin{theorem}\label{thm:surjbis}
 If $L$ and $L'$ are slice, then $\varphi^{L,L'}_k$ induces a surjective map ($k\ge 1$) 
   $$\bfrac{\mathcal{C}(L,L')}{\stackrel{\small{c}}{\sim}}\rightarrow \textrm{Iso}_\textrm{c}\left(N_k \pi_{L'} ,N_k \pi_{L} \right). $$
\end{theorem}
\begin{proof} 
Pick some concordances $X\in \mathcal{C}(O_n,L)$ and $X'\in \mathcal{C}(L',L)$. 
By Lemma \ref{thm:iso}, we have a bijection  $\xi^{c}_X\circ \eta^{c}_{X'}$ from $\bfrac{\mathcal{C}(L,L')}{\stackrel{\small{c}}{\sim}}$ to $\bfrac{\mathcal{C}(O_n)}{\stackrel{\small{c}}{\sim}}$, and the induced identifications of $N_k \pi_L$ and $N_k \pi_{L'}$ with $N_kF_n$ yield a commutative diagram 
 \begin{equation}
 \vcenter{\hbox{\xymatrix{
   \bfrac{\mathcal{C}(L,L')}{\stackrel{\small{c}}{\sim}}\ar[d]_{\xi^{c}_X\circ \eta^{c}_{X'}}^{\textrm{bij.}} \ar[r]^(.38){\, \, \varphi^{L,L'}_k}  & \textrm{Iso}_\textrm{c}\left(N_k \pi_{L'} ,N_k \pi_{L} \right)\ar[d]^{\textrm{bij.}} \\
   \bfrac{\mathcal{C}(O_n)}{\stackrel{\small{c}}{\sim}}\ar[r]^(.38){\, \, \varphi^{O_n}_k} & \textrm{Aut}_\textrm{c}\left(  N_k F_n \right) 
    }}}.
    \label{diag}
\end{equation} 
So we only have to prove the result in the case $L=L'=O_n$, 
i.e. that the bottom horizontal map $\varphi^{O_n}_k$ is surjective. 
By Theorem \ref{thm:conc}, this is actually a consequence of the surjectivity of the map  
$$\varphi^{O_n}_k: \mathcal{C}(O_n)\longrightarrow \textrm{Aut}_\textrm{c}\left(  N_k F_n \right)$$ itself, which can be seen as follows. 
Any element $f$ of $\textrm{Aut}_\textrm{c}\left(  N_k F_n \right)$ is completely characterized by an $n$-tuple of elements in $N_k F_n$, 
which are the conjugating elements $l_i$ from (\ref{eq:conjug}). 
Now, each of these $l_i$ is represented by a word  in the free generators $x_j$, and the techniques of \cite[\S~4]{ABMW} 
can be used to build a welded string link $D_f$ such that $C_f:=\textrm{Tube}(D_f)$ satisfies $\varphi^{O_n}_k(C_f)$. 
(See Section \ref{sec:more} for welded string links and the Tube map.) We do not recall the machinery from \cite{ABMW} here, 
but rather illustrate the procedure for defining $D_f$ on a simple example in Figure \ref{fig:sl2}. 
As the figure suggests, only the successive over-arcs met when running along each strand of $D_f$ is relevant, 
and these can be connected by virtual arcs in any arbitrary way. 
\begin{figure}[h!]
\[
f: \left\{ {\begin{array}{ccl}
 l_1& = & x_3 \\
 l_2& = & 1 \\
 l_3& = & x_2^{-1} x_3^{-1}    
          \end{array}}
    \right.
    \quad \mbox{\larger$\leadsto$} \quad  
    \vcenter{\hbox{\includegraphics[height=2cm]{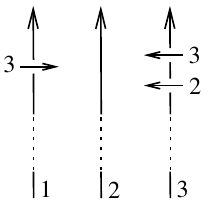}}}
   \quad \mbox{\larger$\leadsto$}  \quad 
   D_f:=\vcenter{\hbox{\includegraphics[height=2cm]{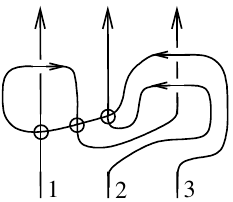}}}
\]
\caption{Building a preimage for $f\in \textrm{Aut}_\textrm{c}\left(  N_k F_n \right)$}\label{fig:sl2}
\end{figure}
\end{proof}

\subsubsection{Proof of Theorem \ref{thm:linkhom} }\label{sec:lhcol}

We first prove the link-homotopy invariance result (i), using the notion of R-coloring for surface diagrams.
Throughout this section, fix systems of meridians $\{m_i\}$ and $\{m'_i\}$ for $L$ and $L'$ as in Definition \ref{def:merid}, hence a basing for all elements of $\mathcal{C}(L,L')$, and more generally elements of $\mathcal{C}^s(L,L')$.

Recall that an immersed surface $S$ in $4$-space can be represented by a \emph{broken surface diagram} $D_S$, as follows. 
Consider the image of $S$ by a projection onto $3$-space; generically, its singular set only contains lines and/or circles of double points, 
which may intersect at triple points and may contain singular  or branch points, 
and the diagram $D_S$ encodes the over/under information along the lines of double points by erasing a neighborhood of the preimages with lower coordinate 
along the projection axis, see \cite{CKS}. Each singular point of $S$ projects to an isolated singular point of $D_S$, lying on a line of double points where the over/under information swaps. 
The \emph{regions of $D_S$} are the connected components of $D_S$ with these singular points removed; in particular, there are locally three regions near a regular double point,\footnote{Here, a regular double point is a double point which is not a triple, singular or branch point. } as illustrated in Figure \ref{fig:W}. 
\begin{figure}[h!]
  \[
\includegraphics[scale=0.8]{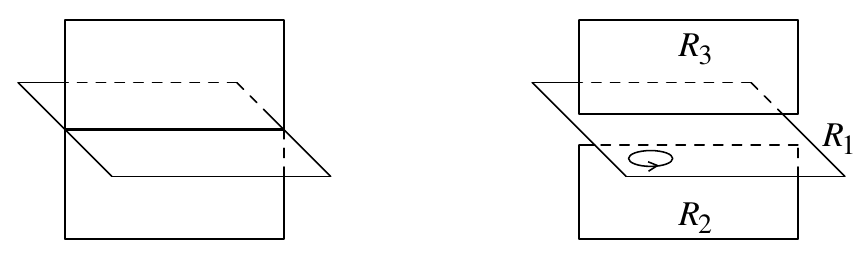}
  \]
  \caption{Generic projection of a surface in $4$-space and the corresponding diagram at a regular double point}
\label{fig:W}
\end{figure}
\\
It is well-known that two broken surface diagrams represent isotopic surfaces in $4$--space if and only if they are connected by a finite sequence of \emph{Roseman moves} \cite{Roseman}. 
This result was extended to surfaces up to (link-)homotopy in \cite{AMW}, by adding three (self-)\emph{singular Roseman moves}. 
\medskip 

Recall that a \emph{surface diagram $G$-coloring} is a map $\sigma$ from the set of all regions of $D_S$ to some group $G$ 
such that near each regular double point, if we denote by $R_1$, $R_2$ and $R_3$ the three regions according to the orientation of $S$ as in Figure \ref{fig:W}, 
we have $\sigma(R_3)= \sigma(R_1)^{-1} \sigma(R_2) \sigma(R_1)$ in $G$;  
see  \cite[\S~4.1.3]{CKS}.
The \emph{color} of a region is simply its image by $\sigma$. 
\medskip

Let $D_L$ be the diagram of the link $L$ obtained by projecting in the direction of the vector $(-v_p)$.\footnote{Recall that $v_p$ is the unit normal vector at the basepoint  $p\in S^2=\partial B^3$; this projection may be assumed to be generic, up to a small perturbation of $p$.}   
Consider a $\pi_L$-coloring of $D_L$ which yields a Wirtinger presentation for $\pi_L$ 
and such that, that for each $i$, the arc of $D_L$ enlaced by the meridian $m_i$ is colored by the element of $\pi_L$ represented by $m_i$.
Let us \emph{fix} an $\textrm{R}\pi_L$-coloring of $D_L$ which is the image of such a  $\pi_L$-coloring by the natural projection $\pi_L\rightarrow \textrm{R}\pi_L$. 

\begin{definition}\label{def:Rcol}
Let $S\in \mathcal{C}^s(L,L')$ be a self-singular concordance from  $L$ to $L'$. 
Let $D_S$ be a  broken surface diagram for $S$, which restricts to $D_L$ on $B^3\times \{0\}$. 
An \emph{R-coloring} of $D_S$ is an $\textrm{R}\pi_L$-coloring which coincides with the fixed coloring for $L$ on the lower boundary. 
Given an R-coloring for $S$, 
the \emph{terminal coloring} is the data of the colors in R$\pi_L$ of the $n$ regions intersecting the meridians $\{m'_i\}$ of $L'$ on the upper boundary.
\end{definition}

It is easily verified that each Roseman or singular Roseman move preserves a given coloring, 
in the sense that it is left unchanged outside a $3$-ball supporting this move, and can be extended in a unique and consitent way inside the $3$-ball. 
Hence we have : 
\begin{claim}\label{claim:col}
Let $S$ and $S'$ be two elements of $\mathcal{C}^s(L,L')$, which admit diagrams $D_S$ and $D_{S'}$ 
that differ by a single Roseman or singular Roseman move. 
Let $B$ be a 3-ball where the local move is supported. 
\item (i)~There is a natural isomorphism from $R\pi_S$ to $R\pi_{S'}$, 
 which is induced from the extension of the identity map outside the 3-ball $B$. 
\item (ii)~An $\textrm{R}\pi_S$-coloring of $D_S$ gives a unique $\textrm{R}\pi_{S'}$-coloring of $D_{S'}$ such that the colorings outside $B$ correspond by the isomorphism from $R\pi_S$ to $R\pi_{S'}$. 
\end{claim}
In particular, since in our context ambient isotopies and link-homotopies are assumed to fix the boundary, we observe that the terminal coloring is preserved by any (singular) Roseman move.
\medskip 

Any element  $C$ of $\mathcal{C}(L,L')$ admits an R-coloring, called \emph{Wirtinger R-coloring}, as follows. 
It is well-known that a Wirtinger presentation for the fundamental group $\pi_C$ of the complement of $C$ can be given from a broken surface diagram \cite[\S~3.2.2]{CKS}.  
Since by (\ref{eq:R}) we know that $\textrm{R}\pi_C$ is isomorphic to $\textrm{R}\pi_L$, 
taking the image of each generator of $\textrm{R}\pi_C$ yields the desired R-coloring.
Moreover, by definition, the invariant $\varphi^{L,L'}_{\textrm{R}}(C)$ can be entirely recovered from this Wirtinger R-coloring: 
the values $\varphi_{C,R}(m'_i)\in \textrm{R}\pi_L$ are simply given by the terminal coloring.

The following lemma tells us that, actually, the terminal coloring of \emph{any} R-coloring of any surface diagram determines $\varphi^{L,L'}_{\textrm{R}}$. 
\begin{proposition}\label{lem:unik}
There is a unique terminal coloring, for any surface diagram of a given element of $\mathcal{C}(L,L')$. 
\end{proposition}
\begin{proof}
Suppose that there exists at least two R-colorings with two different terminal colorings, 
for some surface diagram of $C\in \mathcal{C}(L,L')$. 
Then $\overline{C}\cdot C$ is an element of $\mathcal{C}(L')$ with a different coloring at $B^3\times \{0\}$ and $B^3\times \{1\}$. 
But by Lemma \ref{lem:inverse}, we know that  $\overline{C}\cdot C$ is link-homotopic to $L'\times [0,1]$, 
which cannot admit different colorings at the top and bottom boundaries.
This leads to a contradiction,  since $\overline{C}\cdot C$ and $L'\times [0,1]$ are related by a sequence of Roseman and singular Roseman moves, 
and since each of these moves preserves the colorings at the boundary. 
\end{proof}
\begin{corollary}\label{cor:terminal}
 Let $C\in \mathcal{C}(L,L')$.  
 The terminal coloring of any surface diagram of $C$ determines $\varphi^{L,L'}_{\textrm{R}}(C)$.
\end{corollary}

We can now state the following link-homotopy invariance result, which in particular contains Theorem \ref{thm:linkhom}~(i).
\begin{theorem}\label{thm:linkhombis}
The map $\varphi^{L,L'}_{\textrm{R}}$ is  a link-homotopy invariant.
\end{theorem}
\begin{proof}
Suppose that two elements $C$ and $C'$ of $\mathcal{C}(L,L')$ are link-homotopic, and pick surface diagrams $D_C$ and $D_{C'}$ for these elements. 
Consider some R-coloring for $D_C$. 
By \cite{AMW}, there is a finite sequence of Roseman and self-singular Roseman moves from $D_C$ to $D_{C'}$. Each of these moves preserves the terminal coloring, 
thus showing that $D_{C'}$ has same terminal coloring as $D_C$. 
Since by Corollary \ref{cor:terminal}, the terminal coloring determines $\varphi^{L,L'}_{\textrm{R}}$, this concludes the proof.
\end{proof}

We now observe that, in the case where $L$ and $L'$ are slice links, the induced map 
   $$\varphi^{L,L'}_{\textrm{R}}: \bfrac{\mathcal{C}(L,L')}{\stackrel{\small{lh}}{\sim}}\rightarrow \textrm{Iso}_\textrm{c}\left(\textrm{R}\pi_{L'},\textrm{R}\pi_{L} \right) $$
is a bijection. This follows, on one hand, from the fact that we have a commutative diagram similar to (\ref{diag}), as 
 \begin{equation*}
 \vcenter{\hbox{\xymatrix{
   \bfrac{\mathcal{C}(L,L')}{\stackrel{\small{lh}}{\sim}}\ar[d]_{\xi^{lh}_X\circ \eta^{lh}_{X'}}^{\textrm{bij.}} \ar[r]^(.38){\, \, \varphi^{L,L'}_{\textrm{R}}}  & \textrm{Iso}_\textrm{c}\left(\textrm{R} \pi_{L'},\textrm{R} \pi_{L} \right)\ar[d]^{\textrm{bij.}} \\
   \bfrac{\mathcal{C}(O_n)}{\stackrel{\small{lh}}{\sim}}\ar[r]^(.38){\, \, \varphi^{O_n}_{\textrm{R}}} & \textrm{Aut}_\textrm{c}\left(  \textrm{R} F_n \right) 
    }}}
\end{equation*} 
and, on the other hand, the fact that the lower horizontal map  $\varphi^{O_n}_{\textrm{R}}$
in this diagram is an isomorphism, as proved in  \cite{AMW}. 
Hence we have the following, which implies Theorem \ref{thm:linkhom}~(ii).
\begin{theorem}\label{thm:linkhombis2}
Let $L$ and $L'$ be slice links. 
The map $\varphi^{L,L'}_{\textrm{R}}$ is a complete link-homotopy invariant for $\mathcal{C}(L,L')$.
\end{theorem}

\section{Milnor invariants of concordances}\label{sec:Milnor}

As (\ref{eq:conjug}) shows, the Artin-type invariants defined in Section \ref{sec:phik} are completely determined by the choice of a longitude for each component of $C$. 
In this section, we explain how to extract numerical invariants of embedded concordances from these longitudes, by adapting the definition of Milnor link invariants of \cite{Milnor2} to the $4$-dimensional setting.

As before, throughout this section, $L$ and $L'$ are two concordant $n$-component links, $C$ is an element of $\mathcal{C}(L,L')$, and $\{m_i\}$ and $\{m'_i\}$ are systems of meridians for $L\subset (B^3\times \{0\})$ and $L'\subset (B^3\times \{1\})$, respectively. 

\subsection{Review of Milnor invariants for links}\label{sec:review}

Before we proceed with the definition of Milnor invariants of concordances, it is useful to recall the original construction for links, in the case of the link $L$.

Let $k\ge 1$ be fixed. 
Consider a parallel copy of the $i$th component of the link $L$, having linking number zero with it; 
using our fixed system of meridians for $L$, this defines unambiguously an element $\lambda_i\in N_k \pi_{L}$, 
called the $i$th preferred longitude of $L$. 
By \cite[Thm.~4]{Milnor2}, $\lambda_i$ can be expressed as a word in the (images of the) meridians $m_j$, which we still denote by $\lambda_i$. 

Recall that the \emph{Magnus expansion} $E(\lambda_i)$ of $\lambda_i$ is the formal power series in the noncommuting variables $X_1,\cdots,X_n$ obtained by substituting $1+X_j$ for $m'_j$, and $1-X_j+X_j^2-X_j^3+\cdots$ for $(m'_j)^{-1}$, for all $j$ \cite{MKS}.  

For each sequence $I=i_1 i_2 ...i_{k-1} i$ of (possibly repeating) elements in $\{1,\cdots,n\}$,
we denote by $\mu_{L}(I)$ the coefficient of $X_{i_1}...X_{i_{k-1}}$ in the Magnus expansion $E(\lambda_i)$. 
We also denote by $\Delta_{L}(I)$ the greatest common divisor of the integers $\mu_{L}(J)$, for all sequences $J$ obtained from $I$ by deleting at least one index and permuting cyclicly.

It is shown in \cite{Milnor2} that the residue class $\overline{\mu}_{L}(I)$ of the integer $\mu_{L}(I)$ modulo $\Delta_{L}(I)$
is an isotopy invariant of $L$, for any sequence $I$ of $\le k$ integers. 
These invariants are called \emph{Milnor invariants} of the link $L$. 

It is well-known that, for two concordant links $L$ and $L'$, 
then $\Delta_{L}(I)=\Delta_{L'}(I)$ for any sequence $I$, and all Milnor $\overline{\mu}$-invariants coincide, see \cite{Stallings}.

\subsection{4-dimensional Milnor invariants}\label{sec:4dimM}

We now return to $C\in \mathcal{C}(L,L')$, equipped with the basing $\{m_i\}$ and $\{m'_i\}$. 
Recall from Section \ref{sec:longitude} that 
an $i$th longitude of $C$ is an arc  running, in the boundary of a neighborhood of the $i$th component, from the fixed point $p_{i,0}$ in $(B^3\times \{0\})\setminus L$ to the fixed point $p_{i,1}$ in $(B^3\times \{1\})\setminus L'$.  (Recall that these fixed points are inherited from the choice of basing.) 

Now, there are many possible choices of $i$th longitude for $C$, for our given basing. 
Specifically, 
if $l_i$ and $l'_i$ are two elements of $N_k \pi_{L}$ given by two choices of $i$th longitude for $C$, 
then there exists $u,v\in \mathbb{Z}$ such that 
  \begin{equation}\label{claim:longitude}
    l'_i =  (m_i)^u (\lambda_i)^v l_i ,  
  \end{equation}
where $\lambda_i$ is given, as above, by an $i$th preferred longitude of $L$. 

A \emph{preferred $i$th longitude} of $C$ is an $i$th longitude whose canonical closure has 
linking number zero with the $i$th component of $C$ (meaning that its closure is null-homologous in the complement of the $i$th component). 
Unlike for links in $3$--space, there is not a unique preferred longitude for a given component:  
by (\ref{claim:longitude}), various preferred longitudes may differ by copies of $\lambda_i$. 

For each $k\ge 1$, the choice of a preferred $i$th longitude for $C$ induces an element of $N_k \pi_C$, 
and we denote by $l^0_i\in N_k \pi_{L}$ its image by the isomorphism $(\iota_{L}^k)^{-1}$ of (\ref{eq:Nk}); 
this element of $N_k \pi_{L}$ can be expressed as a word in the meridians $m_j$, also denoted by $l^0_i$. 

Let $I=i_1 i_2 ...i_{m} i$ be a sequence of elements in $\{1,\cdots,n\}$. 
We denote by 
$\mu^{(4)}_C(I)$ the coefficient of $X_{i_1}\cdots X_{i_{m}}$ in the Magnus expansion $E(l^0_i)$,  
and by $\overline{\Delta}_L(I)$ the greatest common divisor of $\mu_L(I)$ and $\Delta_{L}(I)$. 
Notice that $\overline{\Delta}_L(I)=\overline{\Delta}_{L'}(I)$, since $L$ and $L'$ are concordant. 

The main result of this section is Theorem \ref{thm:invmu}, stated in the introduction: 
for any sequence $I$ of at most\footnote{Since $k$ can be chosen arbitrarily, note that this condition is not restrictive. } $k$ elements in $\{1,\cdots,n\}$, 
the residue class $\overline{\mu}^{(4)}_C(I)$ of the integer $\mu^{(4)}_C(I)$ modulo $\overline{\Delta}_{L}(I)$ 
is an (ambient isotopy relative to the boundary) invariant of $C$.

\begin{definition}
The above invariants are called \emph{Milnor $\overline{\mu}^{(4)}$-invariants}. 
The \emph{length} of the invariant is the number of indices in the corresponding sequence. 
\end{definition}

\begin{remark}\label{rem:ark}
Since all Milnor invariants $\overline{\mu}(I)$ and all $\Delta(I)$ vanish for slice links, it follows from Theorem  \ref{thm:invmu} that, 
if $L$ and $L'$ are slice, the integer $\mu^{(4)}_C(I)$ is an invariant of $C$ for any sequence $I$.
\end{remark}

\begin{remark}
In the special case $L=L=O_n$, there is a canonical choice of basing, see \cite[\S 2.2.1.1]{ABMW}, 
and we recover in this way the invariants defined in \cite{ABMW,AMW}. 
As a matter of fact, the invariants defined and studied here are a vast generalization of those defined in \cite[\S~5.1]{ABMW} and \cite[\S~4.1]{AMW}. 
\end{remark}

\begin{proof}[Proof of Theorem \ref{thm:invmu}]
We must prove that the residue class $\overline{\mu}^{(4)}_C(I)$ is independent of the choices made in the construction. 
There are two types of choices to discuss here: the choice of preferred longitudes of $C$, and the choice of a word representing these longitudes in $N_k \pi_{L}$. 
For the latter point, recall that we fixed here a basing,
meaning in particular that we have a generating set $\{m_i\}$ for $N_k \pi_{L}$.
By \cite[Thm.~4]{Milnor2}, two words representing a given preferred longitude of $C$ may
only differ by terms in $\Gamma_k F(m_i)$, where $F(m_i)$ denotes the free group generated by $\{m_i\}$, 
and/or copies of conjugates of $[m_j,\lambda_j]$, where $\lambda_j$ is a word representing the $j$th preferred longitude of the link $L$. 
Following the proof of Theorem 5 in \cite{Milnor2}, we have that these ambiguities 
are controlled by the indeterminacy $\Delta_{L}(I)$. 
More precisely, we use the fact  that the Magnus expansion $E$ satisfies the following: 
\begin{itemize}
\item  for an element $x\in \Gamma_k F(m_i)$, we have $E(x)=1+\textrm{terms of degree $\ge k$}$; hence  Milnor invariants indexed by a sequence of at most $k$ indices remain unchanged, 
\item  $E([m_j,\lambda_j])-1\in D_i$, where $D_i$ is the ideal introduced in \cite{Milnor2}, that controls the indeterminacy for Milnor invariants indexed by a sequence $I$ ending with $i$  -- see \cite[(14)]{Milnor2}. 
\end{itemize} 
It thus remains to verify the independence under the choice of preferred longitudes for $C$. 
As pointed out above, two choices of preferred $i$th longitude for $C$ may differ by a power of  $\lambda_i$. So 
it suffices to observe that, for any $x\in N_k \pi_{L}$, 
the coefficient of any monomial $X_{i_1}\cdots X_{i_{m}}$ in $E(\lambda_i x)$ is given by 
$$ \underbrace{\big( \textrm{coeff. of $X_{i_1}\cdots X_{i_{m}}$ in $E(\lambda_i)$}\big) }_{=\mu_{L}(i_1\cdots i_{m}i)} +\, \big( \textrm{coeff. of $X_{i_{1}}\cdots X_{i_{m}}$ in $E(x)$\big) }  $$ 
$$ + \sum_{j=1}^{m-1} \big( \underbrace{\textrm{coeff. of $X_{i_1}\cdots X_{i_{j}}$ in $E(\lambda_i)$}}_{=\mu_{L}(i_1\cdots i_{j}i)} \big)  \left( \textrm{coeff. of $X_{i_{j+1}}\cdots X_{i_{m}}$ in $E(x)$}\right), $$
thus proving that\footnote{Notice in particular that the first term in the above sum explains why the definition of the indeterminacy $\overline{\Delta}_{L}(I)$ involves $\mu_{L}(I)$.  }
\begin{eqnarray*}
 \big( \textrm{coeff.\,of\,$X_{i_1}\dotsi X_{i_{m}}$ in $E(\lambda_i x)$}\big)  & & \\
 &\hspace{-1cm}\equiv   \big( \textrm{coeff.\,of\,$X_{i_1}\dotsi X_{i_{m}}$ in $E(x)$}\big)  \textrm{ mod }\overline{\Delta}_{L}(i_1\dotsi i_m i). &
\end{eqnarray*}
This completes the proof.
\end{proof}

\subsection{Some properties}

We now give several properties of Milnor $\overline{\mu}^{(4)}$-invariants. 
In particular, we prove Theorem \ref{thm:mu} in the next three subsections. 

\subsubsection{Concordance invariance.} 

The fact that Milnor $\overline{\mu}^{(4)}$-invariants are surface-concordance invariants 
is an almost immediate consequence of (\ref{diagbis}), as for Artin-type invariants. 

\begin{proof}[Proof of Theorem \ref{thm:mu}~(i)]
Let $C$ and $C'$ be two concordances in $\mathcal{C}(L,L')$ with $C \stackrel{c}{\sim} C'$,  
and let $\alpha_i$ be a preferred $i$th longitude of $C$. 
By moving $\alpha_i$ along a concordance $W$ from $C$ to $C'$, 
we obtain a preferred $i$th longitude $\alpha'_i$ of $C'$. 
Hence the commutative diagram~(\ref{diagbis}) implies that 
$(\iota_{L}^k)^{-1}(l_i)=(\iota_{L}^k)^{-1}(l'_i)$ in $N_k\pi_{L}$, 
where $l_i\in N_k\pi_C$ and $l'_i\in N_k\pi_{C'}$ are elements represented by $\alpha_i$ and $\alpha'_i$, respectively.  
\end{proof}

\subsubsection{Link-homotopy invariance}\label{sec:lhcol2}

Let us now show that Milnor $\overline{\mu}^{(4)}$-invariants indexed by non-repeated sequences are link-homotopy invariants, as stated in Theorem \ref{thm:mu}~(ii). For this, we use the diagramatic approach developed in Section \ref{sec:lhcol}.

\begin{proof}[Proof of Theorem \ref{thm:mu}~(ii)]
By \cite{AMW}, if two elements of $\mathcal{C}(L,L')$ are link-homotopic, there is a finite sequence of Roseman and self-singular Roseman moves relating their surface diagrams. 
By (\ref{eq:R}) and Claim~\ref{claim:col}, any element of $\mathcal{C}^s(L,L')$ arising in this sequence has reduced fundamental group isomorphic to R$\pi_L$. 

We are thus led to consider two elements $S$ and $S'$ in $\mathcal{C}^s(L,L')$, with respective surface diagrams $D_S$ and $D_{S'}$, that are related by a single Roseman or self-singular Roseman move, and such that  $\textrm{R}\pi_{S}\cong \textrm{R}\pi_{S'}\cong \textrm{R}\pi_{L}$. \\
Let $B$ be a 3-ball where this local move is supported. 
Pick a preferred $i$th longitude $\alpha^S_i$ of $S$, which represents an element 
$l^S_i$ in R$\pi_{L}$. Suppose that $D_S$ is given an $R$-coloring (see Definition \ref{def:Rcol}).

If $D_S$ and $D_{S'}$  are related by a (usual) Roseman move, 
we can freely assume, up to isotopy, that $\alpha^S_i$ is disjoint from $B$. 
This gives a preferred $i$th longitude $\alpha'^S_i$ of $S$, which is disjoint from the 3-ball $B$, and which yields an $i$th longitude of $S'$ representing the same element $l^S_i\in \textrm{R}\pi_{L}$. 

Now, if $D_S$ and $D_{S'}$  are related by a self-singular Roseman move, 
the longitude  $\alpha^S_i$ may intersect the 3-ball $B$, 
and moving it away from the interior of $B$ may involve passing through a self-singular point; we denote by $\alpha'^S_i$ the resulting preferred $i$th longitude. 
Using the R-coloring of $D_S$,  we can express $l^S_i$ as $u_{i1}^{\varepsilon_1}  u_{i2}^{\varepsilon_2}\cdots u_{ir_i}^{\varepsilon_{r_i}}$, where the $u_{ij}$'s are the colors in R$\pi_l$ of the successive \lq overpassing sheets\rq~ intersecting $\alpha^S_i$ transversally, with intersection number $\varepsilon_{ij}$. 
By Claim~\ref{claim:col}, $\alpha'^S_i$ then represents an element $l'^S_i$ in $R\pi_{L}$ which is equal to 
either $l^S_i$  or 
$$ l'^S_i = v(g^{-1} m_i g) w=vw\left((gw)^{-1} m_i g w\right), $$ 
 for some element $g$ of $R\pi_{L}$, 
where $v=\prod_{i=1}^k u_{ik}^{\varepsilon_k}$ and $v=\prod_{i=k+1}^{r_i} u_{ik}^{\varepsilon_k}$ for some $k$. 

Summarizing, for two link-homotopic elements $C,C'\in \mathcal{C}(L,L')$, we have: 
\begin{claim}
For an element $l_i\in R\pi_{L}$ given by a preferred $i$th longitude of $C$, 
there is an element $l'_i$ given by a preferred $i$th longitude  
of $C'$ such that $l'_i$ is obtained from $l_i$ by multiplying by  
several conjugates of $m_i$.  
\end{claim}

Moreover, recall from \cite[\S~6]{Milnor} that $R\pi_{L}$ has a presentation 
with generating set $\{m_1,...,m_n\}$ and  relation set
\[U=\{[m_j,\lambda_j],[m_j, x^{-1} m_j x]~|~x\in \langle m_1,...,m_n\rangle,~j=1,...,n\}. \]
Hence a representative of $l'_i$ is obtained from one of $l_i$ by multiplying  
by several conjugates of $m_i$ and/or of elements in $U$. 
Now, for $x,y\in \langle m_1,...,m_n\rangle$, the Magnus expansion $E$ satisfies the following: 
\begin{itemize}
\item  $E(x^{-1} m_i x)$ is of the form $1+\left( \textrm{terms containing $X_i$}\right)$, 
\item  $E(y^{-1} [m_j, x^{-1} m_j x] y)$ is of the form $1+$(repeated terms), 
where a monomial $X_{i_1}\cdots X_{i_m}$ is called repeated if $i_j=i_k$ for some $j\neq k$, 
\item  $E(y^{-1} [m_j,\lambda_j] y)-1 \in D_i$, as in the proof of Theorem \ref{thm:invmu}. 
\end{itemize}

Hence, for any non-repeated sequence $j_1\cdots j_si$, we have 
\[\left( \textrm{coeff. of $X_{j_1}\! \cdots\! X_{j_s}$ in $E(l_i)$}\right) \equiv 
\left( \textrm{coeff. of $X_{j_1}\! \cdots\!  X_{j_s}$ in $E(l'_i)$}\right) \textrm{ mod }\overline{\Delta}_L(j_1\! \cdots\!  j_s i). \]
This completes the proof. 
\end{proof}

\subsubsection{Link-homotopy classification: proof of Theorem \ref{thm:mu}~(iii)}

We now show how to use Theorem \ref{thm:linkhombis2} to classify concordances 
up to link-homotopy by $\overline{\mu}^{(4)}$-invariants, in the slice case. 
Recall from Remark \ref{rem:ark} that, in the slice case, the integers $\mu^{(4)}(I)$ are (ambient isotopy fixing the boundary) invariants for all sequences $I$. 
\begin{lemma}\label{lem:phimu}
Let $C,C'$ be two concordances in $\mathcal{C}(L,L')$. 
Suppose that $L$ is a slice link.  Then  $\varphi_{\textrm{R}}^{L,L'}(C)=\varphi_{\textrm{R}}^{L,L'}(C')$ if and only if 
all Milnor $\mu^{(4)}$-invariants indexed by non-repeated sequences coincide for $C$ and $C'$.  
\end{lemma}
\begin{proof}
The fact that $\varphi_{\textrm{R}}^{L,L'}(C)=\varphi_{\textrm{R}}^{L,L'}(C')$ 
means that, for each $i$, there are preferred $i$th longitudes for $C$ and $C'$, represented by elements $l_i$ and $l'_i$ in R$\pi_L$, such that
$l_i^{-1}m_i l_i=(l'_i)^{-1} m_i l'_i$ in  R$\pi_{L}$. This is equivalent to $[m_i, l_i(l'_i)^{-1}]=1\in \textrm{R}\pi_{L}$. 
Since $L$ is a slice link,  $\textrm{R}\pi_{L}$ is isomorphic to the reduced free group R$F_n$.
It is known (see e.g. \cite[Thm.~7.11]{yura}) that $g$ is a representative of $1\in \textrm{R}F_n$ if and only if 
$E(g)$ is of form $1+\left( \textrm{repeated terms}\right)$. 
It follows easily that $\varphi_{\textrm{R}}^{L,L'}(C)=\varphi_{\textrm{R}}^{L,L'}(C')$ if and only if 
\[E(l_i)=E(l'_i)+\left( \textrm{terms containing $X_i$}\right)+\left( \textrm{repeated terms}\right),\]
which is equivalent to having same Milnor $\mu^{(4)}$-invariants indexed by non-repeated sequences. 
 \end{proof}

By combining Theorem \ref{thm:linkhombis2} and Lemma \ref{lem:phimu}, we obtain 
that, if $L$ and $L'$ are slice links, then 
Milnor ${\mu}^{(4)}$-invariants 
indexed by non-repeated sequences form a complete link-homotopy invariant for $\mathcal{C}(L,L')$.
This proves Theorem \ref{thm:mu}~(iii).  
\begin{remark}
It is known (see e.g. \cite[Cor.~4.4.5]{Fenn}) that $g$ is a representative of $1\in N_kF_n$ if and only if 
$E(g)$ is of form $1+\left( \textrm{terms of degree $\geq k$}\right)$. 
This implies that the proof of Lemma \ref{lem:phimu} can be easily adapted to show the following. 
Let $C,C'$ be two concordances in $\mathcal{C}(L,L')$, where $L$ is a slice link.  Then 
$\varphi_k^{L,L'}(C)=\varphi_k^{L,L'}(C')$ if and only if all Milnor $\overline{\mu}^{(4)}$-invariants of length at most $k$ coincide for $C$ and $C'$. 
\end{remark}

\subsubsection{Additivity}

Let $L,L',L''$ be three concordant links, with given systems of meridians. 
We have the following additivity property. 
\begin{proposition}\label{lem:add}
Let $C\in\mathcal{C}(L,L')$ and $C'\in\mathcal{C}(L',L'')$ such that all Milnor $\overline{\mu}^{(4)}$-invariants of $C$, resp. $C'$, of length $\leq l$, resp. $\leq l'$, vanish.  
Then $$\overline{\mu}^{(4)}_{C\cdot C'}(I)\equiv \overline{\mu}^{(4)}_{C}(I)+\overline{\mu}^{(4)}_{C'}(I) \textrm{ mod }\overline{\Delta}_{L}(I) $$ for all sequences $I$ of length $\leq l+l'$.  
\end{proposition}
\begin{proof}
Denote by $\{m_i\}$ and $\{m'_i\}$ a generating system for $N_k \pi_{L}$ and $N_k \pi_{L'}$, respectively, 
given by the systems of meridians. 
For all $j=1,\cdots,n$, denote by $w_j$, resp. $w'_j$, a word in $\{m_i\}$, resp. $\{m'_i\}$, 
representing the image in $N_k \pi_{L}$, resp. in $N_k \pi_{L'}$, 
of a preferred $j$th longitude of $C$, resp. $C'$. 

For a sequence $I$ with final index $j$, Milnor invariant $\overline{\mu}^{(4)}(I)$ of $C\cdot C'$ 
is computed from the Magnus expansion $E$ of a word $W_j$ representing the image in 
$N_k \pi_{L'}$ of a $j$th preferred longitude of $C\cdot C'$.  
We can take $E(W_j)=E(w_j)\cdot E(\tilde{w'_j})$, where 
$E(\tilde{w'_j})$ is obtained from $E(w'_j)$ by substituting 
$E(w_i^{-1} m_i w_i)-1$ for $X_i$, for each $i$.
Since all terms of degree $<l$ in $E(w_i)$ are zero, for each $i$, we have that 
 $$E\left(w_i^{-1} m_i w_i \right) - 1 = X_i + \textrm{ (terms of degree $\ge l$)}. $$
Moreover, since the Magnus expansion of $w'_j$ is the form $E(w'_j)= 1 +($terms of order $\ge l')$, we have that 
  $E(\tilde{w'_j}) = 1 + \sum_{d=l}^{l+l'-1} E_{d}(w'_j) + \left(\textrm{terms of degree $\ge (l+l')$}\right)$,  
where $E_d(x)$ denotes the degree $d$ part of the Magnus expansion of a word $x$.
  It follows that 
  $$  E(W_j) = 1 + \sum_{d<l+l'}\left( E_{d}(w_j) + E_{d}(w'_j) \right) +
\left(\textrm{terms of degree $\ge (l+l')$}\right). $$
Since degree $(l+l')$ terms in $E(W_j)$ contribute to length $(l+l'+1)$ invariants,  this completes the proof. 
\end{proof}

\subsection{Some applications}\label{sec:appli}

In this final section, we give two rather simple applications of our $4$-dimensional  Milnor invariants.

For this, we use the language of welded knotted objects, 
which are a diagrammatic generalization of usual knot diagrams: 
these are planar immersions of some $1$-manifold with transverse double points, which are labeled either as classical crossings 
(as for usual diagrams) or virtual crossings, regarded up to an extended family of Reidemeister moves.  
See Figure \ref{fig:AB} for some examples.
Satoh defined in \cite{Satoh} the Tube map, which turns welded knotted objects into knotted surfaces in $4$-space. 
In particular, the Tube map turns welded string links -- and in particular, welded pure braids -- into elements of $\mathcal{C}(O_n)$; see \cite{ABMW} for details. 
We will not review the definition of this map here, but refer the reader to \cite{Audoux} for a good introduction. 

\subsubsection{More on surface-concordance vs. link-homotopy}\label{sec:more} 

The result of \cite{BT} used in the proof of Lemma \ref{lem:inverse}, saying that concordance implies link-homotopy, gives a surjective group homomorphism 
 $$ \bfrac{\mathcal{C}(L)}{\stackrel{c}{\sim}} \longrightarrow \bfrac{\mathcal{C}(L)}{\stackrel{\small{lh}}{\sim}}. $$
\noindent (More generally, we have a surjective map $  \bfrac{\mathcal{C}(L,L')}{\stackrel{c}{\sim}} \rightarrow \bfrac{\mathcal{C}(L,L')}{\stackrel{\small{lh}}{\sim}}$ 
for two concordant links $L$ and $L'$.) 
We now observe that this epimorphism is not injective in general, 
by exhibiting a nontrivial element of $\bfrac{\mathcal{C}(L)}{\stackrel{c}{\sim}}$ which becomes trivial up to link-homotopy. 

Consider the $2$-component welded pure braids $A$ and $B$ shown on the left-hand side of Figure \ref{fig:AB}. 
Their images by the Tube map, which we still denote by $A$ and $B$, are elements of $\mathcal{C}(O_2)$. 

Recall that, for elements of $\mathcal{C}(O_2)$, all indeterminacies $\overline{\Delta}(I)$ vanish, so that $\mu^{(4)}$-invariants are well-defined over the integers. 
A straightforward computation, using the techniques of \cite{ABMW}, gives that 
\begin{equation}\label{eq:12}
\mu^{(4)}_{A\cdot B}(12)=\mu^{(4)}_{
B\cdot A}(12)=\mu^{(4)}_{A\cdot B}(21)=\mu^{(4)}_{
B\cdot A}(21)=1, 
\end{equation}
while   
\begin{equation}\label{eq:122}
\mu^{(4)}_{A\cdot B}(122)=0\quad \textrm{and} \quad \mu^{(4)}_{B\cdot A}(122)=1. 
\end{equation}
 
By Theorem \ref{thm:mu}~(iii), $\mu^{(4)}(12)$ and $\mu^{(4)}(21)$ form a complete set of link-homotopy invariants for $\mathcal{C}(O_2)$. Hence by (\ref{eq:12}), $A\cdot B$ and $B\cdot A$ are link-homotopic. 
But these two elements are not concordant as a consequence of (\ref{eq:122}) and Theorem \ref{thm:mu}~(i).

\begin{figure}[!h]
\includegraphics[scale=0.85]{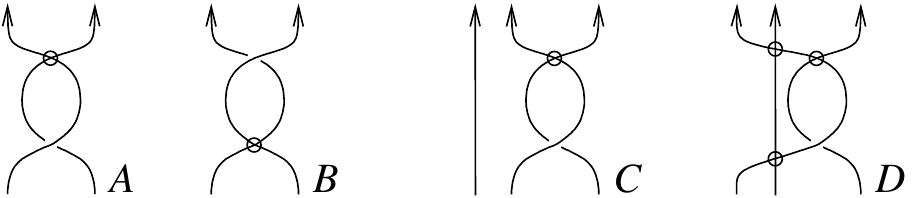}
\caption{The welded pure braids $A$, $B$, $C$ and $D$.} \label{fig:AB}
\end{figure}

\subsubsection{Non abelian group structures}

We showed in Proposition \ref{thm:groups} that, for any link $L$, the quotient monoids 
$\bfrac{\mathcal{C}(L)}{\stackrel{c}{\sim}}$,   
$\bfrac{\mathcal{C}(L)}{\stackrel{\small{lh}}{\sim}}$  and $\bfrac{\mathcal{C}^s(L)}{\stackrel{\small{lh}}{\sim}}$
are all groups.
We now oberve that, even in the simplest case where $L$ is slice, these groups are non-abelian. More precisely, we have the following.
\begin{proposition}\label{proposition1}
For an $n$-component slice link  $L$ ($n\ge 2$),  the group $\bfrac{\mathcal{C}(L)}{\stackrel{c}{\sim}}$ is non-abelian. 
\end{proposition}
\begin{proof}
Consider again the welded braids $A$ and $B$ of Figure \ref{fig:AB}. 
The computation in (\ref{eq:122}) readily implies that $\bfrac{\mathcal{C}(O_n)}{\stackrel{c}{\sim}}$ is not abelian for any $n\ge 2$.  
By Proposition \ref{thm:groups}~(ii), this holds more generally for $\bfrac{\mathcal{C}(L)}{\stackrel{c}{\sim}}$ for any slice link $L$. 
\end{proof}
As noted above, (\ref{eq:12}) shows that $A$ and $B$ commute up to link-homotopy. 
As a matter of fact, one can show that $\bfrac{\mathcal{C}(O_n)}{\stackrel{lh}{\sim}}$ is an abelian group for $n=2$ (this essentially follows from \cite[Prop.~4.11]{ABMW}). 
But this is no longer true for $n\ge 3$: 
\begin{proposition}\label{proposition2}
For an $n$-component slice link $L$ ($n\ge 3$), the groups $\bfrac{\mathcal{C}(L)}{\stackrel{lh}{\sim}}$ and $\bfrac{\mathcal{C}^s(L)}{\stackrel{\small{lh}}{\sim}}$ are non-abelian. 
\end{proposition}
\begin{proof}
Consider the $3$-component welded pure braids $C$ and $D$ shown on the right-hand side of Figure \ref{fig:AB}, and denote their respective images by the Tube map, by the same letter. 
Since these are elements of $\mathcal{C}(O_3)$, all indeterminacies $\overline{\Delta}(I)$ vanish and another simple computation gives that 
$$\mu^{(4)}_{C\cdot D}(123)=0\quad \textrm{and} \quad\mu^{(4)}_{D\cdot C}(123)=1. $$ 
As a consequence, $\bfrac{\mathcal{C}(O_n)}{\stackrel{lh}{\sim}}$ is not abelian for any $n\ge 3$.  
Moreover, since $C$ and $D$ are elements of $\mathcal{C}^s(O_n)$, the fact that  
$C\cdot D\stackrel{lh}{\nsim} D\cdot C$ more generally shows that $\bfrac{\mathcal{C}^s(O_n)}{\stackrel{lh}{\sim}}$ is not abelian for any $n\ge 3$.  
By Proposition \ref{thm:groups}~(iii), we have that the same holds for $\bfrac{\mathcal{C}(L)}{\stackrel{lh}{\sim}}$ and $\bfrac{\mathcal{C}^s(L)}{\stackrel{\small{lh}}{\sim}}$ for any slice link $L$ with $n\ge 3$ components.
\end{proof}

\bibliographystyle{abbrv}
\bibliography{paper}

\end{document}